\documentclass{amsart}

\usepackage{apacite}
\usepackage{natbib}
\usepackage{comment}
\usepackage{hyperref}
\usepackage{color}

\usepackage{amssymb,amsmath}

\usepackage{kbordermatrix}
\newtheorem{thm}{Theorem}
\newtheorem{cor}{Corollary}
\newtheorem{lem}{Lemma}

\newtheorem{prop}{Proposition}
\newtheorem{exmp}{Example}
\newtheorem{defn}{Definition}

\newtheorem{claim}{Claim}
\newtheorem{fact}{Fact}
\newtheorem{con}{Conjecture}

\newcommand{\R}{\mathbb R}

\newcommand{\we}{\ensuremath{\prec_W}}
\def\row#1#2{{#1}_1,\,\ldots,{#1}_{#2}}
\def\brow#1#2{{\bf {#1}}_1,{\bf #1}_2,\,\ldots,{\bf {#1}}_{#2}}

\def\bsum#1#2{{\bf {#1}}_1+{\bf {#1}}_2+\cdots
+{\bf {#1}}_{#2}}
\begin{document}


\title{Simplicial Complexes Obtained from Qualitative Probability Orders}

\author{Paul H. Edelman}
 \email{paul.edelman@vanderbilt.edu}

\author{ Tatyana Gvozdeva}
\email{t.gvozdeva@math.auckland.ac.nz}

 \author{ Arkadii Slinko}
 \email{a.slinko@auckland.ac.nz}



%
%

\maketitle

\begin{abstract}
In this paper we inititate the study of abstract simplicial
complexes which are initial segments of qualitative probability
orders. This is a natural class that contains the threshold
complexes and is contained in the shifted complexes, but is equal
to neither.  In particular we construct a  qualitative probability
order on 26 atoms that has an initial segment which is not a
threshold simplicial complex. Although 26 is probably not the
minimal number for which such example exists we provide some
evidence that it cannot be much smaller. We prove some necessary
conditions for this class and make a conjecture as to a
characterization of them.  The conjectured characterization relies
on some ideas from cooperative game theory.
\end{abstract}


\section{Introduction}
\label{intro}

The concept of qualitative (comparative) probability takes its
origins in attempts of de Finetti (\citeyear{dF}) to axiomatise
probability theory. It also played an important role in the
expected utility theory of  \citet[p.32]{Sa}. The essence of a
qualitative  probability is that it does not give us numerical
probabilities but instead provides us with the information, for
every pair of events, which one is more likely to happen.  The
class of qualitative probability orders is  broader than the class
of probability measures for any $n\ge 5$ \citep{KPS}. Qualitative
probability orders on finite sets are now recognised as an
important combinatorial object \citep{KPS,PF1,PF2} that finds
applications in areas as far apart from probability theory as the
theory of Gr\"obner bases \citep[e.g.,][]{Mac}.

Another important combinatorial object, also defined on a finite
set is an abstract simplicial complex. This is a set of subsets of
a finite set, called faces, with the property that a subset of a
face is also a face. This concept is dual to the concept of a
simple game whose winning coalitions form a set of subsets of a
finite set   with the property that if a coalition is winning,
then every superset of it is also a winning coalition. The most
studied class of simplicial complexes is the class of threshold
simplicial complexes. These arise when we assign weights to
elements of a finite set, set a threshold and define faces as
those subsets whose combined weight is not achieving the
threshold.

Given a qualitative probability order one may obtain a simplicial
complex in an analogous way. For this one has to choose a
threshold---which now will be a subset of our finite set---and
consider as faces all subsets that are earlier than the threshold
in the given qualitative probability order. This initial segment
of the qualitative probability order will, in fact, be  a
simplicial complex.  The collection of complexes arising as
initial segments of probability orders  contains threshold
complexes and is contained in the well-studied class of shifted
complexes \citep{Klivans05,Klivans07}. A natural question is
therefore to ask if this is indeed a new class of complexes
distinct from both the threshold complexes and the shifted ones.

In this paper we give an affirmative answer to both of these
questions. We present an example of a shifted complex on $7$
points that is not the initial segment of any qualitative
probability order.  On the other hand we also construct  an
initial segment of a qualitative  probability order  on $26$ atoms
that is not threshold. We also show that such example cannot be
too small, in particular, it is unlikely that one can be found on
fewer than $18$ atoms.

The structure of this paper is as follows. In Section 2 we
introduce the basics of qualitative probability orders. In Section
3 we consider abstract simplicial complexes and give necessary and
sufficient conditions for them being threshold. In Section 4 we
give a construction that will further provide us with examples of
qualitative probability orders that are not related to any
probability measure.  Finally in Sections 5 and 6 we present our
main result which is an example of a qualitative  probability
order  on $26$ atoms that is not threshold. Section 7 concludes
with a conjectured characterization  of initial segment complexes
that is motivated by work in the theory of cooperative games.


\section{Qualitative  Probability Orders and Discrete Cones}

In this paper all our objects are defined on the set $[n]=\{1,2,\ldots, n\}$. By $2^{[n]}$ we denote the set of all subsets of $[n]$.
An order\footnote{An order in this paper is any reflexive, complete and transitive
binary  relation. If it is also anti-symmetric, it is called linear order.} $\preceq$ on $2^{[n]}$ is called a {\em qualitative probability order} on ${[n]}$ if
\begin{equation}
\label{nontriv}
 \emptyset \preceq A
\end{equation}
for every nonempty subset $A$ of ${[n]}$, and $\preceq$ satisfies de Finetti's axiom,
namely for all
$A,B,C\in 2^{[n]}$
\begin{equation}
\label{deFeq}
A\preceq B \ \Longleftrightarrow \ A\cup C\,\preceq\, B\cup C
\  \mbox{whenever}\ \ (A\cup B)\cap C=\emptyset\,.
\end{equation}

Note that if we have a probability measure ${\bf p}=(\row pn)$ on ${[n]}$,
where $p_i$ is the probability of $i$, then we know the
probability $p(A)$ of every event $A$ and $p(A)=\sum_{i\in A}p_i$.
We may now define a relation $\preceq$ on $2^{[n]}$ by
\[
A\preceq B \quad \mbox{if and only if}\quad  p(A)\le p(B);
\]
obviously $\preceq$ is a qualitative probability order  on
${[n]}$, and any such order is called {\em representable} \citep[e.g.,][]{PF1,GR}.
Those not obtainable in this way are called {\em non-representable}.
 The class of qualitative probability orders is  broader
than the class of probability measures for any $n\ge 5$ \citep{KPS}.
A non-representable qualitative probability order $\preceq$ on ${[n]}$
is said to {\em almost agree} with the measure ${\bf p}$ on ${[n]}$ if
\begin{equation}
\label{almost-repr}
A\preceq B \Longrightarrow p(A)\le p(B).
\end{equation}
If such a measure ${\bf p}$ exists, then the order $\preceq$
is said to be {\em almost representable}. Since the arrow in (\ref{almost-repr})
is only one-sided it is perfectly possible for an almost representable order to
have $A\preceq B$ but not $B\preceq A$ while $p(A)=p(B)$. \par\medskip

We begin with some  standard properties of qualitative probability orders which we will need subsequently.
Let $ \preceq$ be a qualitative probability order on $2^{[n]}$. As usual the following two relations can be derived from it. We write $A\prec B$ if $A\preceq B$ but not $B\preceq A$ and $A\sim B$ if $A\preceq B$ and $B\preceq A$.

\begin{lem}
\label{strict-equiv}
Suppose that $ \preceq$ is a qualitative probability order on $2^{[n]}$, $A, B, C, D \in 2^{[n]}$, $A \preceq B$, $C \preceq D$ and $B \cap D = \emptyset$. Then $A \cup C \preceq B \cup D$. Moreover,  if $A \prec B$ or $C \prec D$, then $A \cup C \prec B \cup D$.
\end{lem}

\begin{proof}
Firstly, let us consider the case when $A\cap C=\emptyset$. Let $B'=B\setminus C$ and $C'=C\setminus B$ and $I=B\cap C$. Then  by  (\ref{deFeq})  we have
\[
A\cup C'\preceq B\cup C'=B'\cup C\preceq B'\cup D
\]
where we have $A\cup C' \prec B'\cup D$ if $A\prec B$ or $C\prec D$. Now we have
\[
A\cup C'\preceq B'\cup D \Leftrightarrow A\cup C=(A\cup C')\cup I\preceq (B'\cup D) \cup I=B\cup D.
\]
Now let us consider the case when $A\cap C\ne \emptyset$. Let $A'=A\setminus C$. By (\ref{nontriv}) and (\ref{deFeq})  we now have
$A'\prec B$. Since now we have $A'\cap C=\emptyset$ so by the previous case
\[
A\cup C=A'\cup C\prec B\cup C\preceq B\cup D.
\]
In this case we always obtain a strict inequality.
\end{proof}
A weaker version of this lemma  can be found in \cite{Mac}[Lemma 2.2].

\begin{defn}
\label{cancel}
A sequence of subsets $(\row Aj; \row Bj)$ of $[n]$ of even length $2j$ is said to be a {\em trading transform} of length $j$ if for every $i\in [n]$
$$
\left|\{k\mid i\in A_k\}\right|=\left|\{k\mid i\in B_k\}\right|.
$$
In other words, sets $\row Aj$ can be converted into $\row Bj$ by rearranging their elements.
We say that an order $\preceq $ on $2^{[n]}$ satisfies the $k$-th cancellation condition $CC_k$ if there does not exist a trading transform
$(\row Ak;\row Bk)$ such that $A_i\preceq B_i$ for all $i\in [k]$ and $A_i\prec B_i$ for at least one $i\in [k]$.
\end{defn}

The key result of \cite{KPS} can now be reformulated as follows.

\begin{thm}[Kraft-Pratt-Seidenberg]
\label{KPStheorem}
A qualitative probability order $\preceq $ is representable if and only if it satisfies $CC_k$ for all $k=1,2,\ldots $.
\end{thm}

It was also shown in \citet[Section 2]{PF1} that $CC_2$
and $CC_3$  hold for linear qualitative probability orders. It follows
from de Finetti's axiom and properties of linear orders.  It can be shown that a qualitative probability order satisfies $CC_2$ and $CC_3$ as well.  Hence $CC_4$
is the first nontrivial cancellation condition.
As  was noticed in \cite{KPS}, for $n<5$ all qualitative probability
orders are representable, but for $n=5$ there are non-representable ones. For $n=5$ all orders are still
almost representable \cite{PF1} which is no longer true for $n=6$ \cite{KPS}. \par\medskip

\medskip

It will be useful for our constructions to rephrase some of these conditions in vector language. To every such linear order $\preceq$, there corresponds a
{\em discrete cone} $C(\preceq)$ in $T^n$, where
$T=\{-1,0,1\}$, as defined in \cite{PF1}.

\begin{defn} A subset $C\subseteq T^n$ is said to be a discrete cone
if the following properties hold:
\begin{enumerate}
\item[{\rm D1.}] $\{{\bf e}_1,\ldots, {\bf e}_n\}\subseteq C$ and $\{-{\bf e}_1, \ldots, -{\bf e}_n\} \cap C = \emptyset$,
where $\{{\bf e}_1,\ldots,{\bf e}_n\}$ is the standard basis of $\R^n$,
\item[{\rm D2.}] $\{-{\bf x},{\bf x}\}\cap C\ne\emptyset$ \
for every ${\bf x}\in T^n$,
\item[{\rm D3.}] ${\bf x}+{\bf y}\in C$ whenever ${\bf x},{\bf y}\in C$ and
${\bf x}+{\bf y}\in T^n$.
\end{enumerate}
\end{defn}
We note that  \cite{PF1} requires ${\bf 0}\notin C$  because
his orders are anti-reflexive. In our case, condition D2 implies
${\bf 0}\in C$.
\medskip

Given a qualitative probability order $\preceq $ on $2^{[n]}$, for every pair of subsets $A,B$ satisfying $B \preceq A$
we construct a characteristic vector of this pair $\chi(A,B)=\chi(A)\!-\!\chi(B)\in
T^n$. We define the set $C(\preceq)$ of all characteristic
vectors $\chi(A,B)$, for $A,B\in 2^{[n]}$ such that $B\preceq A$.
The two axioms of qualitative probability  guarantee that $C(\preceq)$ is  a
discrete cone \citep[see][Lemma~2.1]{PF1}.

Following \cite{PF1},  the cancellation conditions can be reformulated as
follows:

\begin{prop}
A qualitative probability order $\preceq  $ satisfies the $k$-th
cancellation condition $CC_k$ if and only if  there does not exist
a set $\{\brow xk\}$ of nonzero vectors in $C(\preceq )$ such that
\begin{equation}
\label{axm}
\bsum xk={\bf 0}
\end{equation}
and $-{\bf x}_i \notin C(\preceq) $ for at least one $i$.
\end{prop}

Geometrically, a qualitative probability order $\preceq$ is representable if
and only if there exists a positive  vector
${\bf u}\in \R^n$ such that
\[
{\bf x}\in C(\preceq)  \Longleftrightarrow ({\bf u},{\bf x})\ge 0
\quad \mbox{for all}\, \ {\bf x}\in T^n\setminus\{{\bf 0}\},
\]
where $(\cdot,\cdot)$ is the standard inner product;
that is, $\preceq$ is representable if and only if every non-zero
vector in the cone $C(\preceq)$ lies in the closed half-space
$H^{+}_{\bf u}=\{{\bf x}\in\R^n\mid ({\bf u},{\bf x})\ge 0\}$ of the corresponding hyperplane
$H_{\bf u}=\{{\bf x}\in\R^n\mid ({\bf u},{\bf x})= 0\}$.
\par

Similarly, for a non-representable but {\em almost} representable
qualitative  probability order $\preceq$, there exists a  vector ${\bf u}\in \R^n$ with non-negative entries
such that
\[
{\bf x}\in C(\preceq )  \Longrightarrow ({\bf u},{\bf x})\ge 0
\quad \mbox{for all}\, \ {\bf x}\in T^n\setminus\{{\bf 0}\}.
\]
In the latter case we can have ${\bf x}\in C(\preceq)$ and $-{\bf x}\notin C(\preceq)$ despite $({\bf u},{\bf x})= 0$.

In both cases, the normalised vector ${\bf u}$ gives us the probability
measure, namely
${\bf p}=(u_1+\ldots +u_n)^{-1}\left(\row un \right)$, from which
$\preceq $ arises or with which it almost agrees.\par


\section{Simplicial complexes and their cancellation conditions}
\label{cancond}
In this section we will introduce the objects of our study, simplicial complexes that arise as initial segments of a qualitative probability order. Using cancellation conditions for simplicial complexes, we will show that this class contains the threshold complexes and is contained in the shifted complexes.  Using only these conditions it will be easy to show that the initial segment complexes are strictly contained in the shifted complexes.  Showing the strict containment of the threshold complexes will require more elaborate constructions which will be developed in the rest of the paper.

A subset $\Delta \subseteq 2^{{[n]}}$ is an {\em (abstract) simplicial complex} if it satisfies the condition:
\[\text{if } B \in \Delta \text{ and } A \subseteq B, \text{ then } A \in \Delta.\]
Subsets that are in $\Delta$ are called {\em faces}.  Abstract simplicial complexes arose from geometric simplicial complexes in topology \citep[e.g.,][]{Ma}. Indeed, for every geometric simplicial complex $\Delta$ the set of vertex sets of simplices in $\Delta$ is an abstract simplicial complex, also called the {\em vertex scheme} of $\Delta$.
In combinatorial optimization various  abstract simplicial complexes  associated with finite graphs (\cite{JJ}) are studied, such as the independence complex, matching complex etc.
Abstract simplicial complexes are also in one-to-one correspondence with {\em simple games} as defined by \cite{vNM:b:theoryofgames}. A simple game is a pair $G=([n],W)$, where $W$ is a subset of the power set $2^{[n]}$ which satisfies the monotonicity condition:
\[
\text{ if $X\in W$ and $X\subseteq Y\subseteq [n]$, then $Y\in W$.}
\]
 The subsets from $W$ are called {\em winning coalitions} and the subsets from $L=2^{[n]}\setminus W$ are called {\em losing coalitions}. Obviously the set of losing coalitions $L$ is a simplicial complex. The reverse is also true: if $\Delta $ is a simplicial complex, then the set $2^{[n]}\setminus \Delta $ is a set of winning coalitions of a certain simple game.

A well-studied class of simplicial complexes is the {\em threshold} complexes (mostly as an equivalent concept to the concept of a weighted majority game but  also as threshold hypergraphs \citep{RRST}). A simplicial complex $\Delta$ is a  threshold complex if there exist non-negative reals $w_1, \ldots, w_n$ and a positive constant $q$, such that
\[
A \in \Delta \Longleftrightarrow w(A) = \sum_{i \in A} w_i < q.
\]
The same parameters define a  {\em weighted majority game} by setting
\[
A \in W \Longleftrightarrow w(A) = \sum_{i \in A} w_i \ge q.
\]
This game has the standard notation $[q;\row wn]$.\par\medskip

A much larger but still well-understood class of simplicial complexes are {\em shifted} simplicial complexes \citep{Klivans05,Klivans07}. A simplicial complex  is shifted if there exists an  order $\trianglelefteq $ on the set of vertices $[n]$ such that for any face $F$, replacing any of its vertices $x\in F$  with a vertex $y$ such that $y \trianglelefteq x$ results in a subset $(F\setminus \{x\})\cup \{y\}$ which is also a face. Shifted complexes correspond to complete\footnote{sometimes also called linear} games \citep{FMDAM}. A complete game has an order $\trianglelefteq$ on players such that if a coalition $W$ is winning, then replacing any player $x\in W$ with a player $x \trianglelefteq z$ results in a coalition $(W\setminus \{x\})\cup \{z\}$ which is also winning.

A related concept is the so-called Isbel's desirability relation $\le_I$ \cite{tz:b:simplegames}. Given a game $G$ the relation $\le_I $ on $[n]$ is defined by setting $j \le_I i$ if for every set $X\subseteq [n]$ not containing $i$ and~$j$
\begin{equation}
\label{isbel}
X\cup \{j\}\in W \Longrightarrow X\cup \{i\} \in W.
\end{equation}
The idea is that if $j \le_I i$, then $i$ is more desirable as a coalition partner than $j$. The  game is  complete if and only if $\le_I$ is an  order on $[n]$. 
\par\medskip

Let $\preceq $ be a qualitative probability order on $[n]$ and $T\in 2^{[n]}$. We denote
\[
\Delta(\preceq, T)=\{X\subseteq [n]\mid X\prec T\},
\]
where $X\prec Y$ stands for $X\preceq Y$ but not $Y\preceq X$, and call it an {\em initial segment} of $\preceq $.

\begin{lem} Any initial segment of a qualitative probability order is a simplicial complex.
\end{lem}
\begin{proof}Suppose that $\Delta=\Delta(\preceq,T)$ and $B \in \Delta$.  If $A \subset B$, then let $C=B \setminus A$. By (\ref{nontriv}) we have that $\emptyset \preceq C$ and since $A \cap C= \emptyset$ it follows from (\ref{deFeq}) that $\emptyset \cup A \preceq C \cup A$ which implies that $ A \preceq B$. Since $\Delta$ is an initial segment,  $B \in \Delta$ and $A \preceq B$ implies that $A \in \Delta$ and thus $\Delta$ is a simplicial complex.
\end{proof}

We will refer to simplicial complexes that arise as initial segments of some qualitative probability order as an {\em initial segment complex}.
\par\medskip

In a similar manner as  for the qualitative probability orders, cancellation conditions will play a key role in our analyzing simplicial complexes.

\begin{defn}
A simplicial complex $\Delta$ is said to satisfy $CC_k^{*}$ if for no $k\ge 2$ there exists a trading transform  $(A_1, \ldots, A_k;B_1, \ldots, B_k)$, such that $A_i \in \Delta$ and $B_i \notin \Delta$, for every $i \in [k]$.
\end{defn}

Let us show the connection between $CC_k$ and $CC_k^{*}$.
\begin{thm}
Suppose $\preceq$ is a  qualitative probability order  on $2^{[n]}$ and ${\Delta(\preceq, T)}$ is its initial segment. If $\preceq$ satisfies $CC_k$ then ${\Delta(\preceq, T)}$ satisfies $CC_k^{*}$.
\end{thm}

This gives us some initial properties of initial segment complexes. Since conditions $CC_k$, $k=2,3$, hold for all qualitative probability orders \citep{PF1} we obtain

\begin{thm}\label {cc*>3}
If an abstract simplicial complex $\Delta \subseteq 2^{[n]}$ is an initial segment complex, then it satisfies $CC_k^{*}$ for all $k \le 3$.
\end{thm}
From this theorem we get the following corollary, due to Caroline Klivans (personal communication):
\begin{cor} Every initial segment complex is a shifted complex.  Moreover, there are shifted complexes that are not initial segment complexes.
\end{cor}
\begin{proof} Let $\Delta$ be a non-shifted simplicial complex. then it is known to contain an obstruction of the form: there are $i,j \in [n]$, and $A,B \in \Delta$, neither containing $i$ or $j$, so that $A \cup i$ and $B \cup j$ are in $\Delta$ but neither $i \cup B$ nor $j \cup A$ are in $\Delta$ \citep{Klivans05}. But then $(A \cup i,B\cup j;B\cup i,A\cup j)$ is a trading transform that violates $CC_2^{*}$. Since all initial segments satisfy $CC_2^{*}$ they must all be shifted.

On the other hand, there are shifted complexes that fail to satisfy $CC_2^{*}$ and hence can not be initial segments. Let $\Delta$ be the smallest shifted complex (where shiftingis with respect to the usual ordering) that contains $\{1,5,7\}$ and $\{2,3,4,6\}$ Then it is easy to check that neither $\{3,4,7\}$ nor $\{1,2,5,6\}$ are in $\Delta$ but \begin{equation}
(\{1,5,7\},\{2,3,4,6\};\{3,4,7\},\{1,2,5,6\})
\end{equation}
is a transform in violation of $CC_2^{*}$.
\end{proof}

 Similarly, the {\em terminal segment}
\[
G(\preceq, T)=\{X\subseteq [n]\mid T\preceq X\}
\]
of any qualitative probability order is a complete simple game.

The Theorem~2.4.2 of the book \cite{tz:b:simplegames} can be reformulated to give necessary and sufficient conditions for the simplicial complex to be a threshold.

\begin{thm}
\label{all_cc*}
An abstract simplicial complex $\Delta \subseteq 2^{[n]}$ is a threshold complex if and only if the condition $CC_k^{*}$ holds for all $k \ge 2$.
\end{thm}
  Above we showed that the initial segment complexes are strictly contained in the shifted complexes.  What is the relationship between the initial segment complexes and threshold complexes?

\begin{lem} Every threshold complex is an initial segment complex.
\end{lem}
\begin{proof} The threshold complex defined by the weights $w_1, \ldots, w_n$ and a positive constant $q$ is the initial segment of the representable qualitative probability order whose where $p_i=w_i,\  1 \leq i\leq n$ and where the threshold set $T$ has the property that $w(A) \leq w(T) <q$ for all $A \in \Delta$.
\end{proof}

This leaves us with the question of whether this containment is
strict, i.e., are there initial segment complexes which are not
threshold complexes.  One might think that some  initial segment
of a non-representable qualitative probability order is not
threshold. Unfortunately that may not be the case.

\begin{exmp} This example, adapted from \citep{Mac}[Example2.5, Example 3.9] gives a
non-representable qualitative probability order for which every
initial segment complex is threshold. Construct a representable
qualitative probability order on $2^{[5]}$ using the $p_i's$
$\{7,10,16,20,22\}$. The order begins
\begin{equation}
\emptyset \prec 1 \prec 2 \prec 3 \prec 12 \prec 4 \prec 5\prec \cdots
\end{equation}
 where $1$ denotes the singleton set $\{1\}$ and by $12$ we mean $\{1,2\}$. Since the qualitative probability order is representable, every initial segment is a threshold complex. Now suppose we interchange the order of $12$ and $4$. The new ordering, which begins
\begin{equation}
 \emptyset \prec 1 \prec 2 \prec 3 \prec 4 \prec 12 \prec 5 \prec \cdots ,
\end{equation}
is still a qualitative probability order but it is no longer representable \citep[Example 2.5]{Mac}. With one exception, all of the initial segments in this new non-representable qualitative order are initial segments in the original one and thus are threshold. The one exception is the segment
\begin{equation}
\emptyset \prec 1 \prec 2 \prec 3 \prec 4
\end{equation}
which is obviously a threshold complex.
\end{exmp}

  Another approach to finding an initial segment complex that is not threshold is to construct a
 complex that violates $CC_k^{*}$ for some small value of
$k$.  As noted above, all initial segment complexes satisfy
$CC_2^{*}$ and $CC_3^{*}$ so the smallest condition that could
fail is $CC_4^{*}$.
We will now show that for small values of $n$ cancellation condition $CC^{*}_4$ 
is satisfied for any initial segment. This will also give us invaluable information on how to construct a non-threshold initial segment later.

\begin{defn}
Two pairs of subsets $(A_1,B_1)$ and $(A_2,B_2)$ are said to be compatible if the following two conditions hold:
\begin{align*}
& x\in A_1\cap A_2 \Longrightarrow x\in B_1\cup B_2,\ \text{and}\\
& x\in B_1\cap B_2 \Longrightarrow x\in A_1\cup A_2.
\end{align*}
\end{defn}

\begin{lem}
\label{oneless}
Let $\preceq$ be a qualitative probability order on $2^{[n]}$, $T\subseteq [n]$, and let $\Delta=\Delta_n(\preceq, T)$ be the respective initial segment.
Suppose $(\row As,\row Bs)$ is a trading transform and $A_i\prec T\preceq B_j$ for all $i,j\in [s]$. If any two pairs $(A_i,B_k)$ and $(A_j,B_l)$ are  compatible, then $\preceq$ fails to satisfy $CC_{s-1}$.
\end{lem}

\begin{proof}
Let us define
\begin{align}
&\bar{A}_i=A_i\setminus (A_i\cap B_k), \qquad\qquad \bar{B}_k=B_k\setminus (A_i\cap B_k),\\
&\bar{A}_j=A_j\setminus (A_j\cap B_l), \qquad\qquad \bar{B}_l=B_l\setminus (A_j\cap B_l).
\end{align}
We note that
\begin{equation}
\label{emptyintersection}
\bar{A}_i\cap \bar{A}_j=\bar{B}_k\cap \bar{B}_l=\emptyset.
\end{equation}
Indeed, suppose, for example, $x\in \bar{A}_i\cap \bar{A}_j$, then also $x\in A_i\cap A_j$ and by the compatibility $x\in B_k$ or $x\in B_l$. In both cases it is  impossible for $x$ to be in $x\in \bar{A}_i\cap \bar{A}_j$. We note also that by Lemma~\ref{strict-equiv} we have
\begin{equation}
\label{pcompjoined}
\bar{A}_i\cup \bar{A}_j\prec \bar{B}_k\cup \bar{B}_l.
\end{equation}
Now we observe that
\begin{equation*}
(\bar{A}_i,\bar{A}_j, A_{m_1},\ldots, A_{m_{s-2}};\bar{B}_k,\bar{B}_l, B_{r_1},\ldots, B_{r_{s-2}}).
\end{equation*}
is a trading transform. Hence, due to (\ref{emptyintersection}),
\begin{equation*}
(\bar{A}_i\cup \bar{A}_j, A_{m_1},\ldots, A_{m_{s-2}};\bar{B}_k\cup \bar{B}_l, B_{r_1},\ldots, B_{r_{s-2}})
\end{equation*}
is also a trading transform. This violates $CC_{s-1}$ since (\ref{pcompjoined}) holds and $A_{m_t}\prec B_{r_t}$ for all $t=1,\ldots, s-2$.
\end{proof}



By definition of a trading transform we are allowed to use repetitions of the same coalition in it. However we will show that to violate  $CC^{*}_{4}$ we need a trading transform $(A_1, \ldots, A_4; B_1, \ldots, B_4)$ where all $A$'s and $B$'s are different.

\begin{lem}\label{norepet}
Let $\preceq$ be a qualitative probability order on $2^{[n]}$, $T\subseteq [n]$, and let $\Delta=\Delta_n(\preceq, T)$ be the respective initial segment.
Suppose $(\row A4,\row B4)$ is a trading transform and $A_i\prec T\preceq B_j$ for all $i,j\in [4]$. Then
\[|\{\row A4\}|=|\{\row B4\}| =4.\]
\end{lem}
\begin{proof}
Note that every pair $(A_i, B_j),(A_l,B_k)$ is not compatible. Otherwise by Lemma~\ref{oneless} the order $\preceq$ fails $CC_3$, which contradicts to the fact that every qualitative probability satisfies $CC_3$. Assume, to the contrary, that we have at least two identical coalitions among $\row A4$ or $\row B4$. Without loss of generality we can assume $A_1=A_2$.  Clearly all $A$'s or all $B$'s cannot coincide and there are at least two different $A$'s and two different $B$'s. Suppose $A_1 \neq A_3$ and $B_1 \neq B_2$.  The pair $(A_1, B_1), (A_3, B_2)$ is not compatible. It means one of the following two statements is true: either there is $x \in A_1 \cap A_3$ such that $x \notin B_1 \cup B_2$ or there is $y \in B_1 \cap B_2$ such that $y \notin A_1 \cup A_3$. Consider the first case the other one is similar.  We know that $x \in A_1 \cap A_3$ and we have at least three copies of $x$ among $\row A4$. At the same time $x \notin B_1 \cup B_2$ and there could be at most two copies of $x$ among $\row B4$. This is a contradiction.
\end{proof}

\begin{thm}
\label{cc4}
$CC_4^{*}$ holds for $\Delta=\Delta_n(\preceq, T)$ for all $n\le 17$.
\end{thm}

\begin{proof}
Let us consider the set of column vectors
\begin{equation}
\label{36vectors}
U=\{ {\bf x}\in \mathbb{R}^8\mid x_i\in \{0,1\}\ \text{and}\ x_1+x_2+x_3+x_4=x_5+x_6+x_7+x_8=2\}.
\end{equation}
This set has an involution ${\bf x}\mapsto {\bf \bar{x}}$, where $\bar{x}_i=1-x_i$. Say, if ${\bf x}=(1,1,0,0,0,0,1,1)^T$, then ${\bf \bar{x}}=(0,0,1,1,1,1,0,0)^T$. There are 36  vectors from $U$ which are split into 18 pairs $\{{\bf x}, {\bf \bar{x}}\}$.

Suppose now ${\mathcal T}=(A_1,A_2,A_3,A_4; B_1,B_2,B_3,B_4)$ is a trading transform, $A_i\prec T\preceq B_j$  and no two coalitions in the trading transform coincide.
Let us write the characteristic vectors of $A_1$, $A_2$, $A_3$, $A_4$, $B_1$, $B_2$, $B_3$, $B_4$ as rows of $8\times n$ matrix $M$, respectively.
Since $\preceq $ satisfies $CC_3$, by Lemma~\ref{oneless}  we know that no two pairs $(A_i,B_a)$ and $(A_j,B_b)$ are compatible. The same can be said about the complementary pair of pairs $(A_k,B_c)$ and $(A_l,B_d)$, where $\{a,b,c,d\}=\{i,j,h,l\}=[4]$. We have
\[
A_i \prec B_a,\text{ } A_j \prec B_b,\text{ } A_h \prec B_c,\text{ } A_l \prec B_d,
\]
Since $(A_i,B_a)$ and $(A_j,B_b)$ are not compatible one of the following two statements is true: either there exists $x\in A_i\cap A_j$ such that $x\notin B_a\cup B_b$ or  there exists $y\in B_a\cap B_b$ such that $x\notin A_i\cup A_j$. As $\mathcal T$ is the trading transform in the first case we will also have $x\in B_c\cap B_d$ such that $x\notin A_h\cup A_l$; in the second $y\in A_h\cap A_l$ such that $y\notin B_c\cup B_d$.

Let us consider two columns $M_x$ and $M_y$ of $M$ that corresponds to elements $x,y\in [n]$. The above considerations show that both belong to $U$ and $M_x=\bar{M}_y$.

In particular, if $(i,j,k,l)=(a,b,c,d)=(1,2,3,4)$, then the columns $M_x$ and $M_y$ will be as in the following picture
\newline
\hspace*{6.8cm} $x$\hspace{1.4cm}$y$
\[
M=
\left[
\begin{array}{cc}
\chi(A_1)\\
\chi(A_2)\\
\chi(A_3)\\
\chi(A_4)\\
\chi(B_1)\\
\chi(B_2)\\
\chi(B_3)\\
\chi(B_4)
\end{array}
\right]
=
\left[\begin{array}{ccccccccccccc}
&&&&1&&&&0&&&&\\
&&&&1&&&&0&&&&\\
&&&&0&&&&1&&&&\\
&&&&0&&&&1&&&&\\
\hline
&&&&0&&&&1&&&&\\
&&&&0&&&&1&&&&\\
&&&&1&&&&0&&&&\\
&&&&1&&&&0&&&&
\end{array}\right]
\]
(we emphasize however that we have only one such column in the matrix, not both).
We saw that one pairing of indices $(i,a), (j,b), (k,c), (k,d)$ gives us a column from one of the 18 pairs of $U$. It is easy to see that a vector from every pair of $U$ can be obtained by the appropriate choice of the pairing of indices. This means that the matrix contains at least 18 columns. That is $n\ge 18$.
\end{proof}

  While no initial segment
complex on fewer than $18$ points can fail $CC_4^{*}$, there is
such an example on $26$ points which will show that the initial
segment complexes strictly contain the threshold complexes. The
next three sections are devoted to constructing such an example.
The next section presents a general construction technique for
producing almost representable qualitative probability orders from
representable ones.  This technique will be employed in section 5
to construct our example. Some of the proofs required will be done
in section 6.

\color{black}

\section{Constructing almost representable orders from  nonlinear representable ones}
\label{almost}

Our approach to finding an initial segment complex that is not threshold will be to start with a non-linear representable qualitative probability order and then perturb it so as to produce an almost representable order. By judicious breaking of ties in this new order we will be able to produce an initial segment that will violate $CC_4^{*}$. The language of discrete cones will be helpful and we begin with a technical lemma that will needed in the construction.

\begin{prop}
Let $\preceq $ be a non-representable but almost representable
qualitative probability order which almost agrees with a
probability measure  ${\bf p}$. Suppose that the $m$th
cancellation condition $CC_m$ is violated, and that for some
non-zero vectors $\{\brow xm\}\subseteq C(\preceq )$ the condition
(\ref{axm}) holds, i.e., ${\bf x}_1 + \cdots + {\bf x}_m ={\bf 0}$
and ${\bf x}_i \notin C(\preceq)$ for at least one $i \in [m]$.
Then all of the vectors $\brow xm$ lie in the hyperplane $H_{\bf
p}$.
\end{prop}

\begin{proof}
  First note that for every ${\bf x}\in C(\preceq )$ which does not belong
to $H_{\bf p}$, we have $({\bf p},{\bf x})>0$. Hence the condition
(\ref{axm}) can hold only when all ${\bf x}_i\in H_{\bf p}$.
\end{proof}

We need to understand how we can construct new qualitative probability orders from old ones
so we need the following investigation. Let $\preceq $ be a representable but not linear qualitative probability order which agrees with a
probability measure  ${\bf p}$.

Let
$S(\preceq )$ be the set of all vectors of $C(\preceq )$ which lie in the corresponding hyperplane $H_{\bf p}$.
Clearly, if ${\bf x}\in S(\preceq )$, then $-{\bf x}$  is a vector of $S(\preceq )$  as well. Since in the definition of
discrete cone it is sufficient that only one of these vectors is in $C(\preceq )$ we may try to remove
one of them in order to obtain a new qualitative probability order. The new order  will almost agree with ${\bf p}$ and hence will be at least almost representable. The
big question is: what are the conditions under which a set of vectors can be removed from $S(\preceq )$?

What can prevent us from removing a vector from $S(\preceq )$? Intuitively, we cannot remove a vector if the set comparison
corresponding to it is a consequence of those remaining. We need to consider what a consequence means formally.

There are two ways in which one set comparison might imply another one. The first way is by means of the de Finetti condition.
This however is already built in the definition of the discrete cone as $\chi(A,B)=\chi(A\cup C,B\cup C)$. Another way in which
a comparison may be implied from two other is transitivity. This has a nice algebraic characterisation. Indeed, if $C\prec B\prec  A$, then
$\chi(A,C)=\chi(A,B)+\chi(B,C)$. This leads us to the following definition.

Following \cite{CCS} let us define a restricted sum for vectors in a discrete cone ${C}$. Let ${\bf
u},{\bf v}\in {C}$. Then
\[
{\bf u}\oplus {\bf v}=
\left\{
\begin{array}{cl} {\bf u}+{\bf v}& \text{if   ${\bf u}+{\bf v}\in T^n
$},    \\\
\text{undefined} &  \text{if   ${\bf u}+{\bf v}\notin T^n $}.  \end{array}
\right.
\]
It was shown in \cite[Lemma~2.1]{PF1} that the transitivity of a qualitative probability order is equivalent to closedness of its corresponding discrete cone with respect to the restricted addition (without formally defining the latter). The axiom D3 of the discrete cone can be rewritten as
\begin{enumerate}
\item[{\rm D3.}] ${\bf x}\oplus {\bf y}\in C$ whenever ${\bf x},{\bf y}\in C$ and
${\bf x}\oplus {\bf y}$ is defined.
\end{enumerate}
Note that a restricted sum is not associative.



\begin{thm}[Construction method]
\label{constr}
Let $\preceq $ be a representable non-linear qualitative probability order
which agrees with the probability measure ${\bf p}$.
Let $S(\preceq )$ be the set of all vectors of $C(\preceq )$ which lie in the hyperplane $H_{\bf p}$.
Let $X$ be a subset of $S(\preceq )$ such that
\begin{itemize}
\item $X\cap \{{\bf s},-{\bf s}\}\ne\emptyset$ for every ${\bf s}\in S(\preceq )$.
\item $X$ is closed under the operation of restricted sum.
\end{itemize}
Then $Y=S(\preceq )\setminus X$ may be dropped from $C(\preceq )$, that is $C_Y=C(\preceq )\setminus Y$ is a discrete cone.
\end{thm}

\begin{proof}
We first note that if ${\bf x} \in C(\preceq) \setminus S(\preceq)$ and ${\bf y} \in C(\preceq)$, then  ${\bf x} \oplus {\bf y}$, if defined, cannot be in $S(\preceq)$. So due to closedness of $X$ under the restricted addition all axioms of a discrete cone are satisfied for $C_Y$.
On the other hand, if for some two vectors ${\bf x},{\bf y}\in X$ we have ${\bf x}\oplus {\bf y}\in Y$, then $C_Y$ would not be a discrete cone and we would not be able to construct a qualitative probability order associated with this set.
\end{proof}

\begin{exmp}[Positive example]
\label{ex4}
The probability measure
\[
{\bf p}=\frac{1}{16}(6,4,3,2,1).
\]
defines a qualitative probability order $\preceq $ on $[5]$ (which is better written from the other end):
\[
\emptyset\prec 5\prec 4\prec 3\prec 45\prec 35 \sim 2\prec 25\sim 34\prec 1
\prec 345\sim 24\prec 23\sim 15\prec 245\prec 14\sim 235\ldots .
\]
(Here only the first 17 terms are shown, since the remaining ones can be uniquely reconstructed. See \cite[Proposition~1]{KPS} for details).
There are only four equivalences here
\[
\ 35 \sim 2,\ \ 25 \sim 34,\  \ 23\sim 15\ \ \hbox{and} \ \ 14\sim 235,
\]
and all other follow from them, that is:
\begin{align*}
& 35 \sim 2\ \text{implies}\ 345\sim 24,\ 135\sim 12;\\
& 25 \sim 34\ \text{implies}\ 125 \sim 134;\\
& 23\sim 15\ \text{implies}\ 234 \sim 145;\\
& 14\sim 235\ \text{has no consequences}
\end{align*}
Let  ${\bf u}_1=\chi(2,35)=(0,1,-1,0,-1)$, ${\bf u}_2=\chi(34,25)=(0,-1,1,1,-1)$, ${\bf u}_3=\chi(15,23)=(1,-1,-1,0,1)$ and ${\bf u}_4=\chi(235,14)=(-1,1,1,-1,1)$.
Then
\[
S(\preceq)=\{\pm{\bf u}_1, \pm{\bf u}_2, \pm{\bf u}_3, \pm{\bf u}_4\}
\]
and $X=\{{\bf u}_1, {\bf u}_2, {\bf u}_3, {\bf u}_4\}$ is closed
under the restricted addition as ${\bf u}_i\oplus  {\bf u}_j$ is
undefined for all $i\ne j$. Note that ${\bf u}_i\oplus  { -\bf
u}_j$ is also undefined for all $i\ne j$. Hence we can subtract
from the cone $C(\preceq)$ any non-empty subset $Y$ of $-X=\{-{\bf
u}_1, -{\bf u}_2, -{\bf u}_3, -{\bf u}_4\}$ and still get a
qualitative probability. Since
\[
{\bf u}_1+{\bf u}_2+{\bf u}_3+{\bf u}_4={\bf 0}.
\]
it will not be representable. The new order corresponding to the discrete cone $C_{-X}$ is linear.
\end{exmp}

\begin{exmp}[Negative example]
\label{ex5} A certain qualitative probability order is associated
with the Gabelman game of order 3. Nine players are involved each
of whom we think as associated with a certain cell of a $3\times
3$ square:
\begin{center}
\begin{tabular}{|c|c|c|}
\hline
1 & 2 & 3\\
\hline
4 & 5 & 6\\
\hline
7 & 8 & 9\\
\hline
\end{tabular}
\end{center}
The $i$th player is given a positive weight $w_i$, $i=1,2,\ldots,
9$, such that in the qualitative probability order, associated
with ${\bf w}=(\row w9)$,
\[
147\sim 258\sim 369 \sim 123\sim 456\sim 789.
\]
Suppose that we want to construct a qualitative probability order
$\preceq $ for which
\[
147\sim 258\sim 369 \prec 123\sim 456\sim 789 .
\]
Then we would like to claim that it is not weighted since for the vectors
\begin{align*}
{\bf x}_1&=(0,1,1,-1,0,0,-1,0,0)=\chi(123,147),\\
{\bf x}_2&=(0,-1,0,1,0,1,0,-1,0)=\chi(456,258),\\
{\bf x}_3&=(0,0,-1,0,0,-1,1,1,0)=\chi(789,369)
\end{align*}
we have ${\bf x}_1+{\bf x}_2+{\bf x}_3={\bf 0}$.  Putting the sign $\prec $ instead of $\sim$ between $369$ and $123$ will also automatically imply $147\prec 123$, $258\prec 456$ and $369\prec 789$. This means that we are dropping the set of vectors
$\{-{\bf x}_1, -{\bf x}_2, -{\bf x}_3\}$ from the cone while leaving the set $\{{\bf x}_1, {\bf x}_2, {\bf x}_3\}$ there.  This would not be possible since ${\bf x}_1\oplus {\bf x}_2=-{\bf x}_3$. So every $X \supset \{{\bf x}_1, {\bf x}_2, {\bf x}_3\}$ with $X \cap \{-{\bf x}_1, -{\bf x}_2, -{\bf x}_3\} = \emptyset$ is not closed under $\oplus$.
\end{exmp}


\section{An example of a nonthreshold initial segment of a linear qualitative probability order}

In this section we shall construct an almost representable linear qualitative probability order  $\sqsubseteq $ on $2^{[26]}$ and a subset $T\subseteq [26]$, such that the initial segment $\Delta(\sqsubseteq,T)$ of $\sqsubseteq $ is not a threshold complex as it fails to satisfy the condition $CC^{*}_4$.

The idea of the example is as follows.
We will start with a representable linear qualitative probability order  $\preceq$ on $[18]$ defined by weights $\row w{18}$ and extend it to a representable but nonlinear qualitative probability order  $\preceq' $ on $[26]$ with weights $\row w{26}$. A distinctive feature of $\preceq'$ will be the existence of eight sets $A'_1, \ldots, A'_4$, $B'_1, \ldots, B'_4$ in $[26]$ such that:
\begin{enumerate}
  \item The sequence $(A'_1, \ldots, A'_4;B'_1, \ldots, B'_4)$ is  a trading transform.
  \item The sets $A'_1, \ldots, A'_4$, $B'_1, \ldots, B'_4$ are tied in $\preceq'$, that is,
  \[
  A'_1\sim' \ldots A'_4\sim' B'_1\sim' \ldots \sim' B'_4.
  \]
  \item If any two distinct sets $X,Y \subseteq [26]$ are tied in $\preceq'$, then $\chi(X,Y) =\chi (S,T)$, where $S,T \in \{A'_1, \ldots, A'_4,B'_1, \ldots, B'_4\}$. In other words all equivalences in $\preceq'$ are consequences of $A'_i\sim'  A'_j$, $A'_i\sim'  B'_j$, $B'_i\sim' B'_j$, where $i,j \in [4]$.
  \end{enumerate}

Then we will use  Theorem~\ref{constr} to untie the eight sets and to construct a comparative probability order $\sqsubseteq $ for which
\begin{equation*}
A'_1 \sqsubset  A'_2 \sqsubset A'_3 \sqsubset A'_4  \sqsubset B'_1 \sqsubset B'_2 \sqsubset B'_3 \sqsubset B'_4,
\end{equation*}
where $X \sqsubset Y$ means that $X \sqsubseteq Y$ is true but not $Y \sqsubseteq X$.

This will give us  an initial segment $\Delta(\sqsubseteq , B'_1)$ of the linear qualitative probability order  $\sqsubseteq $, which is not  threshold since  $CC^{*}_4$ fails to hold. \par\medskip

Let $\preceq$ be a representable linear qualitative probability order  on $2^{[18]}$ with weights $w_1, \ldots , w_{18}$ that are linearly independent (over $\mathbb{Z}$)  real numbers  in the interval~$[0, 1]$. Due to the choice of weights, no two distinct subsets $X, Y \subseteq [18]$ have equal weights relative to this system of weights, i.e.,
\[
X\ne Y \Longrightarrow w(X)=\sum_{i\in X}w_i \ne w(Y)=\sum_{i\in Y}w_i.
\]

Let us consider again the set $U$ defined in (\ref{36vectors}).
Let $M$ be a subset of  $U$ with the following properties: $|M|=18$ and  ${\bf x}\in M$ if and only if ${\bf \bar{x}}\notin M$. In other words $M$ contains exactly one vector from every pair into which $U$ is split.  By $M$ we will also denote an $8\times18$ matrix whose columns are all the vectors from $M$ taken in arbitrary order. By $A_1,\ldots, A_4, B_1, \ldots, B_4 $ we denote the sets with characteristic vectors equal to the rows $\row M8$  of $M$, respectively.  The way $M$ was constructed secures that the following lemma is true.

\begin{lem}
\label{propertiesofM}
The subsets $A_1,\ldots, A_4, B_1, \ldots, B_4$ s of $[18]$ satisfy:
\begin{enumerate}
\item $(A_1, \ldots , A_4; B_1 , \ldots , B_4 )$ is a trading transform;
\item for any choice of $i,k,j,m \in [4]$ with $i \neq k$ and $ j \neq m$ the  pair $(A_i, B_j), (A_k, B_m)$ is not compatible.
\end{enumerate}
\end{lem}

We shall now embed  $A_1, \ldots, A_4, B_1 , \ldots, B_4$ into $[26]$ and add new elements to them forming $A'_1, \ldots, A'_4,$ $ B'_1, \ldots, B'_4$ in such a way that the characteristic vectors $\chi(A'_1) ,\ldots, \chi(A'_4),$ $\chi(B'_1),\ldots, \chi(A'_1)$  are the rows $M_1',\ldots,M_8'$ of the following matrix
\begin{equation}
\label{matrix}
M'=\quad \kbordermatrix{
& 1 \ldots 18 & \vrule &19 \text{ } \text{ } 20 \text{ } \text{ } 21 \text{ } \text{ } 22 & \vrule& 23 \text{ } \text{ } 24 \text{ } \text{ } 25 \text{ } \text{ } 26 \\
&\begin{array}{c}\chi(A_1) \\ \chi(A_2) \\ \chi(A_3) \\ \chi(A_4)\\ \end{array} & \vrule & I & \vrule& I \\ \hline
&\begin{array}{c}\chi(B_1) \\ \chi(B_2) \\ \chi(B_3) \\ \chi(B_4) \end{array} & \vrule & J & \vrule& I  } ,
\end{equation}
respectively.
%
 Here $I$ is the $4 \times 4$ identity matrix and  $$J=\left( \begin{array}{cccc}
 0 & 0& 0& 1 \\
 1 & 0& 0& 0 \\
 0 & 1& 0& 0 \\
 0 & 0& 1& 0
 \end{array} \right).$$


Note that if $X$ belongs to $[18]$, it also belongs to $[26]$, so the notation $\chi (X)$ is ambiguous as it may be a vector  from ${\mathbb Z}^{18}$  or from ${\mathbb Z}^{26}$, depending on the circumstances. However the reference set will be always clear from the context  and the use of this notation will create no confusion.

One can see that $(A'_1, \ldots,A'_4;B'_1,\ldots,B'_4)$ is again a trading transform and there are no compatible pairs $(A'_i , B'_j ), (A'_k , B'_m)$, where $i,k,j,m \in [4] \text{ and }i \neq k\ \text{or}\  j \neq m.$ We shall now choose weights  $w_{19}, \ldots, w_{26}$ of new elements $19, \ldots, 26$ in such a way that the sets $A'_1, A'_2, A'_3, A'_4, B'_1, B'_2, B'_3, B'_4$ all have the same weight $N$, which is a sufficiently large number. It will be clear from the proof how large it should be.

To find weights $w_{19}, \ldots, w_{26}$ that satisfy this condition we need to solve the following system of linear equations
\begin{equation}
\label{vesa}
\left( \begin{array}{cc}
 I &I  \\
 J & I
\end{array} \right)
\left( \begin{array}{c}
w_{19} \\
\vdots \\
w_{26}
\end{array} \right)=N\textbf{1}-
M\cdot {\bf w},
\end{equation}
where  $\textbf{1}=(1,\ldots, 1)^T \in {\mathbb R}^8$ and ${\bf w}=(w_1, \ldots, w_{18})^T \in \mathbb{R}^{18}$.

The matrix from (\ref{vesa}) has rank~$7$, and the augmented matrix of the system has the same rank. Therefore, the solution set is not empty, moreover, there is one free variable (and any one can be chosen for this role). Let this free variable be $w_{26}$ and let us give it value $K$, such that $K$ is large but much smaller than $N$. In particular, $126 < K<N$.
Now we can express all other weights $w_{19}, . . . , w_{25}$ in terms of $w_{26} = K$ as follows:
\begin{equation}\label{urav} \begin{split}
w_{19}  =  N - &K -(\chi(A_4)   -  \chi(B_1)+  \chi(A_1)) \cdot {\bf w}  \\
w_{20}  =  N - & K -(\chi(A_4)   -  \chi(B_1) +  \chi(A_1) - \chi(B_2)  + \chi(A_2)  )  \cdot {\bf w} \\
w_{21}  =  N - & K -(\chi(A_4)  -  \chi(B_1)  +  \chi(A_1)- \chi(B_2)  + \chi(A_2) -\\
& \chi(B_3)+ \chi(A_3))\cdot {\bf w} \\
w_{22}  =  N - &K -\chi(A_4) \cdot {\bf w} \\
w_{23}  =  K -&(- \chi(A_4)+\chi(B_1)  )  \cdot {\bf w}\\
w_{24}  =  K -&(  - \chi(A_4)+\chi(B_1) -  \chi(A_1)+  \chi(B_2)  )  \cdot {\bf w}  \\
w_{25}  =  K -&(-  \chi(A_4) +\chi(B_1)  -  \chi(A_1)  +  \chi(B_2) - \chi(A_2) +  \chi(B_3)  )\cdot {\bf w}.
\end{split} \end{equation}

By choice of $N$ and $K$ weights $w_{19}, . . . , w_{25}$ are positive. Indeed, all ``small'' terms in the right-hand-side of~(\ref{urav}) are strictly less then $7 \cdot 18=126<\min\{K,N-K\}$.

Let $\preceq'$ be the representable qualitative probability order  on $[26]$ defined by the weight vector ${\bf w}'=(w_1, \ldots, w_{26})$. Using $\preceq'$
we would like to construct a linear qualitative probability order  $\sqsubseteq$ on $2^{{[26]}}$  that ranks the subsets $A_i'$ and $B_j'$ in the sequence 
\begin{equation}
\label{eightstrict}
A'_1 \sqsubset  A'_2 \sqsubset A'_3 \sqsubset A'_4  \sqsubset B'_1 \sqsubset B'_2 \sqsubset B'_3 \sqsubset B'_4.
\end{equation}

\par\medskip

We will make use of Theorem~\ref{constr} now. Let $H_{{\bf w}'}=\{x \in {\mathbb R}^n | ({\bf w}',x)=0\}$ be the hyperplane with the normal vector ${\bf w}'$ and $S(\preceq')$ be the set of all vectors of the respective discrete cone $C(\preceq')$ that lie in $H_{{\bf w}'}$. Suppose
\[
X'= \{ \chi(C,D) \mid C, D \in \{A'_1, \ldots, A'_4, B'_1, \ldots, B'_4\}\ \text{and $D$ earlier than $C$ in (\ref{eightstrict})}\}.
\]
This is a subset of $T^{26}$, where $T=\{-1,0,1\}$. Let also $Y'=S(\preceq')\setminus X'$. To use Theorem~\ref{constr} with the goal to achieve (\ref{eightstrict}) we need to show, that
\begin{itemize}
  \item $S(\preceq') = X' \cup -X'$ and
  \item $X'$ is closed under the operation of restricted sum.
\end{itemize}
If we could prove this, then $C(\sqsubseteq) = C(\preceq') \setminus Y'$ is a discrete cone of a linear qualitative probability order  $\sqsubseteq $ on $[26]$ satisfying (\ref{eightstrict}). Then the initial segment $\Delta(\sqsubseteq, B'_1)$ will not be a threshold complex, because the condition $CC_4^{*}$ will fail for it. \par\medskip

Let $Y$ be one of the sets $A_1, A_2, A_3, A_4, B_1, B_2, B_3, B_4$. By $\breve{Y}$ we will denote the corresponding superset of $Y$ from the set $\{A'_1, A'_2, A'_3, A'_4, B'_1, B'_2, B'_3, B'_4\} $.

\begin{prop}
\label{closedZ}
The subset
\[
X= \{ \chi(C,D) \mid C,D \in \{A_1, \ldots, A_4, B_1, \ldots, B_4\}\ \text{with $ \breve{D}$ earlier than $\breve{C}$  in (\ref{eightstrict})}\}.
\]
 of $T^{18}$ is closed under the operation of restricted sum.
\end{prop}

\begin{proof}
Let ${\bf u}$ and ${\bf v}$ be any two vectors in $X$. As we will see the restricted sum ${\bf u} \oplus {\bf v}$ is almost always undefined. Without loss of generality we can consider only five cases.\par\smallskip

{\bf Case 1.} ${\bf u}= \chi (B_i, A_j)$ and ${\bf v}=\chi (B_k, A_m)$, where $i \neq k$ and $j \neq m$. In this case by Lemma~\ref{propertiesofM} the pairs $(B_i, A_j)$ and $ (B_k, A_m)$ are not compatible.  It means that there exists $p \in [18]$ such that either $p \in B_i \cap B_k$ and $p \notin A_j \cup A_m$ or $p \in A_j \cap A_m$ and $p \notin B_i \cup B_k$. The vector ${\bf u} + {\bf v}$ has $2$ or $-2$ at $p$th  position and ${\bf u} \oplus {\bf v}$ is undefined.  This is illustrated in the table below:\par\bigskip

\begin{tabular}{|c|c|c|c|c|c|c|c|} \hline
 & $\chi(B_i)$ & $\chi (B_k)$ & $\chi(A_j)$ & $ \chi(A_m)$ & $\chi (B_i, A_j)$ & $\chi (B_k, A_m)$ & ${\bf u}+{\bf v}$\\ \hline
$p$th & 1 & 1& 0& 0 & 1& 1 & 2\\
coordinate & 0 & 0& 1& 1 & -1& -1 & -2\\ \hline
\end{tabular}

\vspace{3mm}

{\bf   Case 2.} ${\bf u}= \chi (B_i, A_j)$, ${\bf v}=\chi (B_i, A_m)$ or ${\bf u}=\chi (B_j, A_i)$, $ {\bf v}=\chi (B_m,A_i)$, where $j \neq m$.
In this case choose $k \in [4]\setminus \{i\}$. Then the pairs $(B_i, A_j)$ and $(B_k, A_m)$ are not compatible. As above, the vector $\chi (B_i, A_j)+\chi(B_k,A_m)$ has $2$ or $-2$ at some position~$p $. Suppose $p \in B_i \cap B_k$ and $p \notin A_j \cup A_m$.  Then $B_i$  has a $1$ in $p$th position and each of the vectors $\chi (B_i, A_j)$ and $\chi (B_i, A_m)$ has a $1$ in $p$th position as well. Therefore, ${\bf u} \oplus {\bf v}$ is undefined because ${\bf u} + {\bf v}$ has $2$ in $p$th position.  Similarly, in the case when $p \in A_j \cap A_m$ and $p \notin B_i \cup B_k$ the $p$th coordinate of ${\bf u} + {\bf v}$ is $-2$.
The case when ${\bf u}=\chi (B_j, A_i)$ and  ${\bf v}=\chi (B_m,A_i)$ is similar.\par\smallskip

{\bf Case 3.} ${\bf u}= \chi (B_i, B_j)$, ${\bf v}=\chi (B_k, B_m)$ or ${\bf u}=\chi (A_i, A_j)$,  ${\bf v}=\chi (A_k,A_m)$, where 
$\{i,j,k,m\}=[4]$.
By construction of $M$ there exists $p\in [18]$ such that $p\in B_i\cap B_k$ and $p\notin B_j\cup B_m$ or $p\notin B_i\cup B_k$ and $p\in B_j\cap B_m$. So there is $p \in [18]$, such that ${\bf u} + {\bf v}$ has $2$ or $-2$ in $p$th position. Thus ${\bf u} \oplus {\bf v}$ is undefined.\par\smallskip

{\bf Case 4.} ${\bf u}= \chi (B_i, B_j)$, ${\bf v}=\chi (B_k, B_m)$ or ${\bf u}=\chi (A_i, A_j)$, $ {\bf v}=\chi (A_k,A_m)$, where $i = k$ or $j = m$.
If $i=k$ and $j=m$, then ${\bf u} \oplus {\bf v}$ is undefined. Consider the case $i=k$, $j \neq m$ and ${\bf u}= \chi (B_i, B_j)$, ${\bf v}=\chi (B_i, B_m)$.
Let $s= [4] \setminus \{i,j,m\}$. By construction of $M$ either we have $p\in [18]$ such that $p\in B_i\cap B_s$ and $p\notin B_j\cup B_m$ or $p\notin B_i\cup B_s$ and $p\in B_j\cap B_m$. In both cases ${\bf u} + {\bf v}$  has $2$ or $-2$ in position $p$. \par\smallskip

{\bf Case 5.} ${\bf u}= \chi (B_i, B_j)$, ${\bf v}=\chi (B_k, B_m)$ or ${\bf u}=\chi (A_i, A_j)$, $ {\bf v}=\chi (A_k,A_m)$, where $ j =k$ or $i = m$.
Suppose $ j =k$. Since $i > j$ and $j> m$ we have $i> m$. This implies that $\chi (B_i, B_m)$ belongs to $X$. On the other hand ${\bf u} + {\bf v} = \chi(B_i) - \chi(B_m)= \chi(B_i, B_m)$. Therefore ${\bf u} \oplus {\bf v}={\bf u} + {\bf v} \in X$.
\end{proof}

\begin{cor}\label{closedX}
$X'$ is closed under restricted sum.
\end{cor}

\begin{proof}
We will have to consider the same five cases as in the Proposition~\ref{closedZ}. As above in the first four cases the restricted sum of vectors will be undefined. In the fifth case, when ${\bf u}= \chi (B'_i, B'_j)$, ${\bf v}=\chi (B'_k, B'_m)$ or ${\bf u}=\chi (A'_i, A'_j)$, $ {\bf v}=\chi (A'_k,A'_m)$, where $ j =k$ or $i = m$, we will have  ${\bf u} + {\bf v} = \chi(B'_i) - \chi(B'_m)= \chi(B'_i, B'_m)\in X'$ or ${\bf u} + {\bf v} = \chi(A'_i) - \chi(A'_m)= \chi(A'_i, A'_m)\in X'$.
\end{proof}

To satisfy conditions of Theorem~\ref{constr} we need also to show that the intersection of the discrete cone $C(\preceq')$ and the hyperplane $H_{{\bf w}'}$ equals to $X' \cup -X'$. More explicitly we need to prove the following:

\begin{prop}
\label{conseq}
Suppose $C,D \subseteq [26]$ are tied in $\preceq'$, that is $C\preceq' D$ and $D\preceq' C$. Then $\chi(C,D) \in X' \cup -X'$.
\end{prop}

\begin{proof}
Assume to the contrary that there are two sets $C, D \in 2^{[26]}$ that have equal weights with respect to the corresponding system of weights defining $\preceq'$  but $\chi(C,D) \notin X' \cup -X'$.
The sets $C$ and $D$  have to contain some of the elements from $[26]\setminus [18]$ since $w_1, \ldots , w_{18}$ are linearly independent. Thus $C = C_1 \cup C_2 \text{ and } D = D_1 \cup D_2$, where $C_1,D_1 \subseteq [18]$ and $C_2, D_2 \subseteq [26]\setminus [18]$ with $C_2$ and $D_2$ being nonempty. We have
\[
0=\chi(C,D)\cdot {\bf w}' = \chi(C_1, D_1)\cdot {\bf w} + \chi (C_2, D_2) \cdot {\bf w}^+,
\]
where ${\bf w}^+=(w_{19}, \ldots, w_{26})^T$. By~(\ref{urav}), we can express weights  $w_{19}, \ldots, w_{26}$ as linear combinations with integer coefficients of $N, K$ and $\row w{18}$ obtaining
\[
\chi (C_2, D_2)\cdot {\bf w}^+ =\left( \sum_{i=1}^4\gamma_i \chi(A_i) + \sum_{i=1}^4\gamma_{4+i} \chi(B_i)\right)\cdot {\bf w} + \beta_1 N + \beta_2 K,
\]
where $\gamma_i, \beta_j \in {\mathbb Z}$.

Clearly the expression in the  bracket on the right-hand-side is just a vector with integer entries. Let us denote it $\alpha$. Then
\begin{equation}
\label{introduction_of_alpha}
\chi (C_2, D_2)\cdot {\bf w}^+ = {\bf \alpha}\cdot {\bf w} + \beta_1 N + \beta_2 K,
\end{equation}
where  ${\bf \alpha} \in {\mathbb Z}^{18}$. We can now write $\chi(C,D)\cdot {\bf w}'$  in terms of ${\bf w}, K$ and $N$:
$$
0=\chi(C,D)\cdot {\bf w}' = (\chi(C_1, D_1) + {\bf \alpha})\cdot {\bf w} +  \beta_1 N + \beta_2 K.
$$
We recap that $K$ was chosen to be much greater then $\sum_{i\in [18]}w_i$ and $N$ is much greater then $K$. So if $\beta_1, \beta_2$ are different from zero then $|\beta_1 N +\beta_2 K |$ is a very big number, which cannot be canceled out by $(\chi(C_1, D_1) + {\bf \alpha})\cdot {\bf w}$.  Weights $w_1, \ldots, w_{18}$ are linearly independent, so for arbitrary ${\bf b} \in Z^{18}$ the dot product ${\bf b}\cdot {\bf w}$ can be zero if and only if ${\bf b}={\bf 0}$. Hence
$$
w(C)=w(D) \mbox{ iff } \chi(C_1, D_1) = -\alpha \text{ and } \beta_1=0, \beta_2=0.
$$
Taking into account that $\chi(C_1,D_1)$ is a vector from $T^{18}$, we get
\begin{equation}\label{uslovie24}
{\bf  \alpha} \notin T^{18} \Longrightarrow w(C) \neq w(D).
\end{equation}

We need the following two claims to finish the proof, their proofs are delegated to the next section.

\begin{claim}\label{from'tocons}
Suppose $\chi(C_1,D_1)$ belongs to $X \cup -X$. Then
$\chi(C,D)$ belongs to $X' \cup -X'$.
\end{claim}

\begin{claim}
\label{alpha}
If ${\bf \alpha} \in T^{18}$, then $\alpha$ belongs to $X\cup -X$.
\end{claim}

Now let us show how with the help of these two claims the proof of Proposition~\ref{conseq} can be completed. The sets $C$ and $D$ have the same weight and this can happen only if ${\bf \alpha}$ is a vector in $T^{18}$. By Claim~\ref{alpha} ${\bf \alpha} \in X \cup -X$. The characteristic vector $\chi(C_1,D_1)$ is equal to $-{\bf \alpha}$, hence $\chi(C_1,D_1) \in X \cup -X$. By Claim~\ref{from'tocons} we get $\chi(C,D) \in X' \cup -X' $, a contradiction.
\end{proof}

\begin{thm}
There exists a  linear qualitative probability order $\sqsubseteq$ on $[26]$ and $T\subset [26]$ such that the initial segment $\Delta(\sqsubseteq,T)$ is not a threshold complex.
\end{thm}

\begin{proof}
By Corollary~\ref{closedX} and Proposition~\ref{conseq} all conditions of Theorem~\ref{constr} are satisfied.  Therefore $C(\preceq') \setminus (-X')$ is a discrete cone $C(\sqsubseteq )$, where $\sqsubseteq$ is a almost representable linear qualitative probability order. By construction $A'_1 \sqsubset A'_2 \sqsubset A'_3 \sqsubset A'_4 \sqsubset B'_1 \sqsubset B'_2 \sqsubset B'_3 \sqsubset B'_4$ and thus $\Delta(\sqsubseteq, B'_1)$ is an initial segment, which is not a threshold complex.
\end{proof}

Note that we have a significant degree of freedom in constructing such an example. The matrix $M$ can be chosen in $2^{18}$ possible ways and we have not specified the linear qualitative probability order $\preceq$.
\section{Proofs of Claim~\ref{from'tocons} and Claim~\ref{alpha}}
Lets fix some notation first. Suppose ${\bf b} \in {\mathbb Z}^{k}$ and  ${\bf x}_i \in {\mathbb Z}^{n} $ for $i \in [k]$. Then we define the product
\[
{\bf b} \cdot ({\bf x}_1, \ldots, {\bf x}_k) = \sum_{i \in [k]} b_i {\bf x}_i.
\]
 It resembles the dot product (the difference is that the second argument is a sequence of vectors) and is denoted in the same way. For a sequence of vectors
$({\bf x}_1, \ldots, {\bf x}_k)$ we also define $({\bf x}_1, \ldots, {\bf x}_k)_p=({\bf x}_1^{(p)}, \ldots, {\bf x}_k^{(p)})$, where ${\bf x}_i^{(j)}$ is the $j$th coordinate of vector  ${\bf x}_i$.

We start with the following lemma.
\begin{lem}
\label{scalmult}
Let ${\bf b}\in \mathbb{Z}^6$. Then
\begin{equation*}
{\bf b}\cdot (\chi(B_1,A_4),\chi(B_2,A_1), \chi(B_3,A_2),\chi(A_2,A_1),\chi(A_3,A_1),\chi(A_4,A_1)) = {\bf 0}
\end{equation*}
 if and only if  ${\bf b}={\bf 0}$.
\end{lem}

\begin{proof}
We know that the pairs $(B_1, A_4)$ and $(B_2, A_1)$ are not compatible. So there exists an element $p$ that lies in the intersection $B_1 \cap B_2$ (or $A_1 \cap A_4$), but $p \notin A_4 \cup A_1$ ($p \notin B_1 \cup B_2$, respectively). We have exactly two copies of every element among $A_1, \ldots, A_4$ and $B_1, \ldots, B_4$. Thus, the element $p$ belongs to $A_2 \cap A_3$ ($B_3 \cap B_4$) and doesn't belong to $B_3 \cup B_4$ ($A_2 \cup A_3$ ). The following table illustrates this:

\vspace{3mm}

\begin{tabular}{|c|c|c|c|c|c|c|c|c|} \hline
& $\chi(A_1)$ & $\chi (A_2)$ & $\chi(A_3)$ & $ \chi(A_4)$ & $\chi (B_1)$ & $\chi (B_2)$ & $\chi(B_3)$ & $\chi(B_4) $\\ \hline
$\text{$p$th} $ & 0 & 1& 1& 0 & 1& 1 & 0& 0\\
$\text{coordinate}  $& 1 & 0& 0& 1 & 0& 0 & 1 & 1\\ \hline
\end{tabular}

\vspace{3mm}

\noindent Then at $p$th position we have
\[
 (\chi(B_1,A_4),\chi(B_2,A_1), \chi(B_3,A_2),\chi(A_2,A_1),\chi(A_3,A_1),\chi(A_4,A_1))_p=\pm (1,1,-1,1,1,0)
\]
and hence
\[
b_1+b_2-b_3+b_4+b_5 = 0.
\]
From the fact that  other pairs are not compatible we can get more equations relating $\row b6$:
\[
\begin{array}{ccc}
b_1- b_2 + b_3 - b_4 - b_6 = 0& \text{from}& (B_1,A_4), (B_3,A_2);\\
-b_1+ b_2 + b_3 + b_5 + b_6 = 0 & \text{from}& (B_1,A_4), (B_4,A_3);\\
 b_2 + b_5 + b_6 = 0& \text{from} &(B_1,A_1), (B_2,A_2);\\
b_4 + b_6 = 0 & \text{from} &(B_1,A_1), (B_3,A_3);\\
 b_3 + b_5 + b_6 = 0&  \text{from}& (B_1,A_1), (B_3,A_2).
\end{array}
\]
The obtained system of linear equations has only the zero solution.
\end{proof}

\begin{lem}
\label{scalarw}
Let ${\bf a}=(\row a8)$ be a vector in ${\mathbb Z}^{8}$ whose every coordinate $a_i$ has absolute value which is at most $100$. Then  ${\bf a}\cdot {\bf w}^+ =0$  if and only if  ${\bf a}={\bf 0} $.
\end{lem}
\begin{proof}
We first rewrite~(\ref{urav})  in more convenient form:
\begin{equation}\label{newurav} \begin{split}
w_{19}&  =  N - K -(-\chi(B_1,A_4)+  \chi(A_1))\cdot {\bf w} \\
w_{20} & =  N - K -(-\chi(B_1,A_4) -\chi(B_2,A_1)+ \chi(A_2) )\cdot {\bf w} \\
w_{21}&  =  N - K -(-\chi(B_1,A_4) -\chi(B_2,A_1) - \chi(B_3,A_2)+  \chi(A_3))\cdot {\bf w} \\
w_{22}&  =  N - K -\chi(A_4) \cdot {\bf w} \\
w_{23}&  =  K - \chi(B_1,A_4)  \cdot {\bf w}\\
w_{24}&  =  K -( \chi(B_1,A_4) + \chi(B_2,A_1) ) \cdot {\bf w} \\
w_{25}&  =  K -(\chi(B_1,A_4)   +\chi(B_2,A_1) + \chi(B_3,A_2) )\cdot {\bf w}\\
w_{26}& =  K
\end{split} \end{equation}

We calculate the dot product  ${\bf a}\cdot {\bf w}^+$  substituting the values of $w_{19}, \ldots, w_{26}$ from~(\ref{newurav}):
\begin{equation}
\label{wsystem}
\begin{split}
0={\bf a}\cdot {\bf w}^+ &=N\sum_{i \in [4]}a_i - K \left(\sum_{i \in [4]}a_i - \sum_{i \in [4] }a_{4+i}\right)  \\
&- \Bigl[  \chi(B_1,A_4)\left(\sum_{i =5}^{7}a_i  - \sum_{i =1}^{3}a_i \right) +\chi(B_2,A_1)\left(\sum_{i =6}^{7}a_i  - \sum_{i =2}^{3}a_i\right) \\
&+ \chi(B_3,A_2)(-a_3 +a_7)+ \sum_{i\in [4]}a_i\chi(A_i)   \Bigr] \cdot {\bf w}.
\end{split}
\end{equation}

The numbers $N$ and $K$ are very big and $\sum_{i \in [18]}w_i$ is small. Also $|a_i|\le 100$. Hence the three summands cannot cancel each other. Therefore  $\sum_{i \in [4]}a_i = 0$ and $\sum_{i \in [4]}a_{4+i} = 0$. The expression in the square brackets should be zero because the coordinates of ${\bf w}$ are linearly independent.


We know that $a_1 = -a_2-a_3-a_4$, so the expression in the square brackets in~(\ref{wsystem}) can be rewritten in the following form:
\begin{equation}
\label{fform}
\begin{split}
b_1\chi(B_1,A_4) +b_2 \chi(B_2,A_1)+b_3 \chi(B_3,A_2)+\\ a_2\chi(A_2,A_1)+a_3\chi(A_3,A_1)+a_4\chi(A_4,A_1),
\end{split}
\end{equation}
where $b_1=\sum_{i =5}^{7}a_i  - \sum_{i =1}^{3}a_i,$ $b_2=\sum_{i =6}^{7}a_i  - \sum_{i =2}^{3}a_i$ and $b_3 = a_7-a_3.$

By Lemma~\ref{scalmult} we can see that expression~(\ref{fform})  is zero iff $b_1=0,$ $b_2=0, b_3=0$ and $ a_2=  0, a_3=0, a_4=0$ and this happens iff ${\bf a} = {\bf 0}$.
\end{proof}

\begin{proof}[Proof of Claim 1]
Assume to the contrary that $\chi(C_1,D_1) \in X \cup -X$ and $\chi(C,D)$ does not belong to $X' \cup -X'$. Consider $\chi(\breve{C_1}, \breve{D_1}) \in X' \cup -X'$. We know that the weight of $C$ is the same as the weight of  $D$, and also that the weight of $\breve{C}_1$ is the same as the weight of $\breve{D}_1$.  This can be written as
\begin{align*}
&\chi(C_1,D_1)\cdot {\bf w} + \chi(C_2,D_2) \cdot {\bf w}^+ = 0, \\
&\chi(C_1,D_1)\cdot {\bf w} + \chi(\breve{C_1}\setminus C_1, \breve{D_1}\setminus D_1 )\cdot {\bf w}^+ = 0.
\end{align*}
We can now see that \[(\chi(\breve{C_1}\setminus C_1, \breve{D_1}\setminus D_1) - \chi(C_2, D_2))\cdot {\bf w}^+ = 0.\]
The left-hand-side of the last equation is a linear combination of weights $w_{19}, \ldots, w_{26}$.
Due to Lemma~\ref{scalarw} we conclude from here that
\[
\chi(\breve{C_1}\setminus C_1,\breve{D_1}\setminus D_1 ) - \chi(C_2,D_2) = \bf 0.
\]
But this is equivalent to  $\chi (C,D) = \chi(\breve{C_1},\breve{D_1}) \in X$, which is a contradiction.
\end{proof}


\begin{proof}[Proof of Claim~\ref{alpha}]
We remind the reader that $\bf\alpha$ was defined in (\ref{introduction_of_alpha}). Sets $C$ and $D$ has the same weight and we established that $\beta_1=\beta_2=0$. So
\begin{equation*}
\label{alpha_equation}
\chi(C_2,D_2) \cdot {\bf w}^+ = {\bf  \alpha}\cdot {\bf w}.
\end{equation*}
If we look at the representation of the last eight weights in~(\ref{newurav}), we note that the weights $w_{19}$, $w_{20}$, $w_{21}$, $w_{22}$ are much heavier than the weights
$w_{23}$, $w_{24}$, $w_{25}$, $w_{26}$. Hence $w(C)=w(D)$ implies
\begin{equation}
\label{equal_heavy}
\begin{split}
|C_2 \cap \{19, 20,21,22\}| = |D_2 \cap \{19, 20,21,22\}| & \text{ and }\\
|C_2 \cap \{23,24,25,26\}| = |D_2 \cap \{23, 24,25,26\}|. &
\end{split}
\end{equation}
That is $C$ and $D$ have equal number of super-heavy weights and equal number of heavy ones.

Without loss of generality we can assume that $C_2 \cap D_2$ is empty.
Similar to derivation in the proof of Lemma~\ref{scalarw}, the vector $\bf \alpha$ can be expressed as
\begin{equation}
\label{alpha-repr-temp}
\alpha = a_1\chi(B_1,A_4) +a_2 \chi(B_2,A_1) + a_3 \chi(B_3,A_2) + \sum_{i\in [4]}b_i \chi(A_i)
\end{equation}
for some $a_i, b_j \in {\mathbb Z}$.
The characteristic vectors $\chi(A_1), \ldots, \chi(A_4)$ participate in the representations of super-heavy elements $w_{19}, \ldots, w_{22}$ only. Hence $b_i=1$ iff element $18+i\in C_2$ and $b_i=-1$ iff element $18+i\in D_2$. Without loss of generality we can assume that $C_2 \cap D_2 = \emptyset$. By~(\ref{equal_heavy}) we can see that  if $C_2$ contains some super-heavy element $p \in \{19, \ldots, 22\}$ with $\chi(A_k)$, $ k \in [4]$, in the representation of  $w_p$, then $D_2$ has a super-heavy $q \in \{19, \ldots, 22\}$, $q \neq p$ with $\chi(A_t), t \in [4] \setminus \{k\}$ in representation of  $w_q$. In such case $b_k= -b_t =1$ and
\[
b_k \chi(A_k)+ b_t \chi(A_t) = \chi(A_k,A_t).
\]
By~(\ref{equal_heavy}) the number of super-heavy element in $C_2$ is the same as the number of super-heavy  elements in $D_2$. Therefore~(\ref{alpha-repr-temp}) can be rewritten in the following way:
\begin{equation}\label{alpha-repr}
\alpha = a_1\chi(B_1,A_4) +a_2 \chi(B_2,A_1) + a_3 \chi(B_3,A_2) +\\
 a_4\chi(A_i,A_p) + a_5\chi(A_k,A_t),
\end{equation}
where $a_1, a_2, a_3 \in {\mathbb Z}$; $a_4,a_5 \in \{0,1\}$ and $\{i,k,t,p\} = [4]$.\par\smallskip

Now the series of technical facts will finish the proof.

\begin{fact}
\label{three}
Suppose ${\bf a}=(a_1,a_2,a_3) \in \mathbb{Z}^3$ and $|\{i, k, t\}|=|\{j, m, s\}|=3$. Then
\[
a_1 \chi(B_j,A_i) + a_2\chi( B_m,A_k) + a_3\chi( B_s,A_t) \in T^{18}
\]
if and only if
\begin{equation}
\label{3-list}
{\bf a} \in \{(0,0,0),\ (\pm1,0,0),\ (0,\pm1,0),\ (0,0,\pm1),\ (1,1,1),\  (-1,-1,-1)\}.
\end{equation}
\end{fact}

\begin{proof}
The pairs $((B_j,A_i), (B_m,A_k))$, $((B_j,A_i), (B_s,A_t))$ and $((B_m,A_k), (B_s,A_t))$  are not compatible. Using the same technique as in the proofs of Proposition~\ref{closedZ} and Lemma~\ref{scalmult} and watching a particular coordinate we get
\[
(a_1 + a_2-a_3),\  (a_1 - a_2+a_3),\  (-a_1 + a_2+a_3) \in T,
\]
respectively.
The absolute value of the sum of every two of these terms is at most two. Add the first term to the third. Then $|2a_2| \le 2$ or, equivalently,  $|a_2|\le 1$. In a similar way we can show that $|a_3|\le 1$ and $|a_1|\le 1$. The only vectors that satisfy all the conditions above are those listed in (\ref{3-list}).
\end{proof}

\begin{fact}
\label{three+}
Suppose ${\bf a}=(a_1,a_2,a_3) \in \mathbb{Z}^3$ and $|\{i, k, t\}|=|\{j, m, s\}|=3$. Then
\[
a_1 \chi(B_j,A_i) + a_2\chi( B_m,A_k) + a_3\chi( B_s,A_t)+ \chi(A_k,A_t) \in T^{18}
\]
if and only if
\begin{equation}
\label{4-list}
{\bf a} \in \{(0,0,0),\  (0,1,0),\  (0,0,-1),\ (0,1, -1)\}.
\end{equation}
\end{fact}

\begin{proof}
Considering non-compatible pairs $((B_m,A_k), (B_s,A_t))$, $((B_j,A_i), (B_m,A_k))$, $((B_j,A_i), (B_{s},A_t))$, $((B_j,A_k), (B_{s},A_i))$, $((B_j,A_t), (B_{m},A_i))$,
we get the inclusions
\begin{equation*}
(-a_1 +a_2+ a_3),\  (a_1+a_2 - a_3 -1),\ (a_1-a_2+ a_3+1),\ (a_1-1),\ (a_1+1) \in T,
\end{equation*}
respectively.
We can see that $|2a_2-1| \le 2$ and $|2a_3+1| \le 2$ and $a_1=0$. So $a_2$ can be only $0$ or $1$ and  $a_3$ can have values $-1$ or $ 0$.
\end{proof}

\begin{fact}
\label{three_and_oneint}
Suppose ${\bf a}=(a_1,a_2,a_3) \in \mathbb{Z}^3$ and $\{i, k, t, p\} = [4]$ and $|\{j, m, s\}|=3$. Then
\[
a_1 \chi(B_j,A_i) + a_2\chi( B_m,A_k) + a_3\chi( B_s,A_t) + \chi(A_i,A_p) \in T^{18}
\]
 if and only if
 \[
 a \in \{(0,0,0),\ (1,0,0),\ (1,1,1),\ (2,1,1)\}.
 \]
\end{fact}

\begin{proof}
 Let $\ell\in [4]\setminus  \{j, m, s\}$. From consideration of the following non-compatible pairs
\begin{align*}
& ((B_j,A_i),(B_m,A_k)),\  ((B_j,A_i), (B_s,A_t)),\  ((B_m,A_k), (B_{s},A_t)),\  ((B_j,A_i), (B_{m},A_t)), \\
& ((B_j,A_i), (B_{m},A_p)),\  ((B_j,A_i), (B_{s},A_p)),\  ((B_s,A_t), (B_{\ell},A_i))
\end{align*}
we get the following inclusions
\begin{align*}
& (a_1 +a_2- a_3-1),\  (a_1-a_2 + a_3 -1),\
(-a_1+a_2+ a_3),\\
& (a_1-1),\
(a_1-a_3),\
(a_1-a_2),\ (a_2-a_3+1)\in T,
\end{align*}
respectively.
So we have $|2a_3-1| \le 2$ (from  the second and the third inclusions) and $|2a_2-1| \le 2$ (from  the first and the third inclusions) from which we immediately get $a_2, a_3 \in \{1,0\}$. We also get $a_1 \in \{2,1,0\}$ (by the forth inclusion).
\begin{itemize}
  \item If $a_1=2$, then by the fifth and sixth inclusions $a_3=1$ and $a_2=1$.
  \item If $a_1 = 1$, then $a_2$ can be either zero or one.  If $a_2=0$ then we have $\chi(B_j,A_i)+ a_3\chi( B_s,A_t) + \chi(A_i,A_p) = \chi(B_j,A_p)+ a_3\chi( B_s,A_t)$. By Fact~\ref{three}, $a_3$ can be zero only. On the other hand, if $a_2=1$, then $a_3=1$ by the seventh inclusion.
  \item If $a_1=0$ then $a_2$ can be a $0$ or a $1$. Suppose $a_2=0$. Then $a_3=0$ by the first two inclusions.  Assume $a_2=1$. Then $a_3=0$  by the third inclusion and on the other hand $a_3=1$ by the second inclusion, a contradiction.
\end{itemize}
This proves the statement.
\end{proof}

\begin{fact}
\label{three_and_two}
Suppose ${\bf a}=(a_1,a_2,a_3) \in \mathbb{Z}^3$ and $\{i, k, t, p\} = [4]$ and $|\{j, m, s\}|=3$.
Then
\[
a_1 \chi(B_j,A_i) + a_2\chi( B_m,A_k) + a_3\chi( B_s,A_t) + \chi(A_i,A_p) + \chi(A_k,A_t)\notin T^{18}.
\]
\end{fact}

\begin{proof}
 Let $\ell\in [4]\setminus  \{j, m, s\}$.
Using the same technique as above from consideration of non-compatible pairs
\begin{align*}
& ((B_j,A_i), (B_{m},A_t)),\  ((B_s,A_t), (B_{j},A_k)),\ ((B_j,A_i), (B_s,A_t)),\\
& ((B_m,A_k), (B_{s},A_t)),\  ((B_j,A_i), (B_{m},A_p)),\ ((B_j,A_i), (B_{\ell},A_k))
\end{align*}
we obtain inclusions:
\[
a_1,\
a_3,\
 (a_1-a_2 + a_3),\
(-a_1+a_2+ a_3),\
(a_1-a_3),\
(a_1-a_3-2) \in T,
\]
respectively.

From the last two inclusions we can see that $a_1-a_3=1$. This, together with the first and the second inclusions, imply $(a_1, a_3) \in \{(1,0), (0,-1)\}$. Suppose $(a_1,a_3)=(1,0)$. Then
\[
\chi(B_j,A_i) + a_2\chi( B_m,A_k) + \chi(A_i,A_p) + \chi(A_k,A_t) =\chi(B_j,A_p) + a_2\chi( B_m,A_k) + \chi(A_k,A_t).
\]
By Fact~\ref{three_and_oneint}, it doesn't belong to $T^{18}$ for any value of $a_2$.

Suppose now that $(a_1,a_3)=(0,-1)$. Then by the third and the forth inclusions $a_2$ can be only zero. Then ${\bf a}=(0,0,-1)$ and
\[
-\chi( B_s,A_t) + \chi(A_i,A_p) + \chi(A_k,A_t) = -\chi( B_s,A_k) + \chi(A_i,A_p).
\]
However, by Fact~\ref{three_and_oneint} the right-hand-side of this equation is not a vector of $T^{18}$.
\end{proof}

\begin{fact}\label{final fact}
Suppose  ${\bf a} \in \mathbb{Z}^5$ and
\[
{\bf v}=a_1 \chi(B_j,A_i) + a_2\chi( B_m,A_k) + a_3\chi( B_s,A_t) + a_4\chi(A_i,A_p) + a_5\chi(A_k,A_t).
\]
If $a_4,a_5 \in \{0,1,-1\}$ and $v \in T^{18}$, then $v$ belongs to $X$ or $-X$.
\end{fact}

\begin{proof}
First of all, we will find the possible values of ${\bf a}$ in case ${\bf v} \in T^{18}$. By Facts~\ref{three} --\ref{three_and_two} one can see that ${\bf v} \in T^{18}$ iff ${\bf a}$ belongs to the set
\begin{align*}
Q=&\{(0,0,0,0,0),\  (\pm1,0,0,0,0),\ (0,\pm1,0,0,0), \ (0,0,\pm1,0,0),\ (1,1,1,0,0), \\
    &(0,0,0,\pm1,0),\ (\pm1,0,0,\pm1,0),\  (\pm1, \pm1, \pm1,\pm1,0),\  (\pm2, \pm1, \pm1,\pm1,0),\\
    &(0,0,0,0,\pm1),\ (0,\pm1,0,0,\pm1),\  (0,0,\mp1,0,\pm1),\ (0,\pm1,\mp1,0,\pm1) \}.
\end{align*}

By the construction of $\preceq$ the sequence $(A_1, \ldots, A_4; B_1, \ldots, B_4)$ is a trading transform. So for every $\{i_1, \ldots, i_4 \} = \{j_1, \ldots, j_4 \} = [4]$ the equation
\begin{equation}\label{trading}
\chi(B_{i_1},A_{j_1}) + \chi(B_{i_2},A_{j_2})+ \chi(B_{i_3},A_{j_3})+\chi(B_{i_4},A_{j_4}) = 0.
\end{equation} holds.
Taking~(\ref{trading}) into account one can show, that for every ${\bf a} \in Q$, vector ${\bf v}$ belongs to $X$ or $-X$. For example, if ${\bf a} = (2,1, 1,1,0)$ then
\begin{multline*}
2 \chi(B_j,A_i) + \chi( B_m,A_k) + \chi( B_s,A_t) + \chi(A_i,A_p) =\\
\chi(B_j,A_i) - \chi(B_{\ell},A_p)+ \chi(A_i,A_p) = \chi(B_j,B_{\ell}),
\end{multline*}
where $\ell\in [4] \setminus \{j,m,s\}$.
 \end{proof}

One can see that ${\bf v}$ from the Fact~\ref{final fact} is the general form of $\bf \alpha$. Hence ${\bf \alpha} \in T^{18}$ if and only if ${\bf \alpha} \in X \cup -X$ which is Claim~2.
\end{proof}

\section{Acyclic games and a conjectured characterization}

So far we have shown that the initial segment complexes strictly contain the threshold complexes and are strictly contained within the shifted complexes.  In this section we introduce some ideas from the theory of simple games to formulate a conjecture that characterizes initial segment complexes.  The idea in this section is to start with a simplicial complex and see if there is a natural linear order available on $2^{[n]}$ which gives a qualitative probability order and has the original complex as an initial segment. We will follow the presentation of Taylor and Zwicker \cite{tz:b:simplegames}.

Let $\Delta \subseteq 2^{[n]}$ be a simplicial complex.  Define
the \emph{Winder desirability relation}, $\leq_W$, on $2^{[n]}$ by
$A \leq_W B$ if and only if for every $Z \subseteq [n]
\setminus((A \setminus B)\cup(B \setminus A))$ we have that
\begin{equation*}
(A \setminus B) \cup Z \notin \Delta \Rightarrow (B \setminus A)
\cup Z \notin \Delta.
\end{equation*}
Furthermore define the \emph{Winder existential ordering}, $\prec_W$, on $2^{[n]}$ to be
\begin{equation*}
A \prec_W B \Longleftrightarrow \textrm{It is not the case that }
B \leq_W A.
\end{equation*}

\begin{defn}\label{d:SA} A simplicial complex $\Delta$ is called strongly acyclic if there are no $k$-cycles
\begin{equation*}
A_1 \we A_2 \we \cdots A_k \we A_1
\end{equation*}
for any $k$ in the Winder existential ordering.
\end{defn}

\begin{thm}\label{t:initsa} Suppose $\preceq$ is a qualitative probability order on $2^{[n]}$ and $T \in 2^{[n]}$. Then the initial segment $\Delta(\preceq,T)$ is strongly acyclic.
\end{thm}
\begin{proof} Let $\Delta=\Delta(\preceq,T)$. It
follows from the definition that $A \we B$ if and only if there
exists a $Z \in [n] \setminus((A \setminus B)\cup (B \setminus
A))$ such that $(A \setminus B) \cup Z \in \Delta$ and $(B
\setminus A) \cup Z \notin \Delta$. From the definition of
$\Delta$ it follows that
\begin{equation*}
(A \setminus B) \cup Z \prec (B \setminus A) \cup Z
\end{equation*}
which,  by de Finneti's axiom \ref{deFeq},
implies
\begin{equation*}
A \setminus B \prec B \setminus A
\end{equation*}
and hence, again  by de Finneti's axiom \ref{deFeq},
\begin{equation*}
 A \prec B.
\end{equation*}
Thus a $k$-cycle
\begin{equation*}
A_1 \we \cdots \we A_k \we A_1
\end{equation*}
in $\Delta$ would imply a $k$-cycle
\begin{equation*}
A_1 \prec \cdots \prec A_k \prec A_1
\end{equation*}
which contradicts that $\prec$ is a total order.
\end{proof}

\begin{con}\label{c:weak} A simplicial complex $\Delta$ is an initial segment complex  if and only if it is strongly acyclic.
\end{con}

We will return momentarily to give some support for Conjecture~\ref{c:weak}. First, however, it is worth noting that the necessary condition of being strongly acyclic from Theorem~\ref{t:initsa} allows us to see that there is little relationship between being an initial segment complex and satisfying the conditions $CC_k^*$.

\begin{cor} For every $M > 0$ there exist simplicial complexes that satisfy $CC_M^*$ but are not initial segment complexes.
\end{cor}
\begin{proof} Taylor and Zwicker \cite{tz:b:simplegames} construct a family of complexes $\{G_k\}$, which they call Gabelman games that satisfy $CC_{k-1}^*$ but not $CC_k^*$.  They then show \cite[Corollary 4.10.7]{tz:b:simplegames} that none of these examples are strongly acyclic. The result then follows from Theorem~\ref{t:initsa}.
\end{proof}

Our evidence in support of Conjecture~\ref{c:weak} is based on the
idea that the Winder existential order can be used to produce the
related qualitative probability order for strongly acyclic
complexes. Here are two lemmas that give some support for this
belief:

\begin{lem} If $\Delta$ is a simplicial complex with $A \in \Delta$ and $B \notin \Delta$ then $A \we B$.
\end{lem}
\begin{proof} Let $Z=A \cap B$. Then
\begin{align*}
(A \setminus B) \cup Z &= A \in \Delta \\
(B \setminus A) \cup Z &=B \notin \Delta,
\end{align*}
and so $A \we B$.
\end{proof}

\begin{lem}\label{l:weprop} For any $\Delta$, the Winder existential order \we satisfies the property
\begin{equation*}
A \we B \Longleftrightarrow A\cup D \we B \cup D
\end{equation*}
for all $D$ disjoint from $A \cup B$.
\end{lem}
\begin{proof}
See \cite[Proposition 4.7.8]{tz:b:simplegames}.
\end{proof}

This pair of lemmas leads to a slightly stronger version of
Conjecture~\ref{c:weak}.

\begin{con}\label{c:strong} If $\Delta$ is strongly acyclic then there exists an extension of \we to a qualitative probability order.
\end{con}

  What are the barriers to proving
Conjecture~\ref{c:strong}?  
The Winder order need not be transitive. In fact there are
examples of threshold complexes for which \we \  is not transitive
\cite[Proposition 4.7.3]{tz:b:simplegames}.  Thus one would have
to work with the transitive closure of \we, which does not seem to
have a tractable description.  In particular we do not know if the
analogue to Lemma~\ref{l:weprop} holds for the transitive closure
of \we.


\section{Conclusion}\label{s:conc} In this paper we have begun the study of a class of simplicial complexes that are combinatorial generalizations of threshold complexes derived from qualitative probability orders. We have shown that this new class of complexes strictly contains the threshold complexes and is strictly contained in the shifted complexes. Although we can not give a complete characterization of the complexes in question, we conjecture that they are the strongly acyclic complexes that arise in the study of cooperative games. We hope that this conjecture will draw attention to the ideas developed in game theory which we believe to be too often neglected in the combinatorial literature.

\bibliographystyle{apacite}
\bibliography{bib}

\end{document}